\documentclass[12pt]{amsart}
\usepackage{amsmath,amsfonts,amssymb}
\usepackage{color,graphics,graphicx,enumerate,setspace,pinlabel}
\usepackage[all]{xy}
\DeclareMathOperator{\im}{im}

\def\R{\mathbb{R}}
\def\Z{\mathbb{Z}}

\def\Q{{\mathbb{Q}}}

\def\C{{\mathbb{C}}}

\def\lmat{\left(\begin{smallmatrix}}
\def\rmat{\end{smallmatrix}\right)}

\def\sm{\setminus}
\def\zt{\Bbb{Z}[t^{\pm 1}]}
\def\wti{\widetilde}
\def\what{\widehat}
\def\ll{\langle}
\def\rr{\rangle}

\def\i{\iota}
\def\ol{\overline}
\def\g{\gamma}
\def\l{\lambda}

\theoremstyle{plain}
\newtheorem{theorem}{Theorem}[section]

\newtheorem{lemma}[theorem]{Lemma}
\newtheorem{corollary}[theorem]{Corollary}

\newtheorem{claim}[theorem]{Claim}

\newcommand{\eps}{\varepsilon}

\makeindex
\usepackage{graphicx}
\usepackage{caption}
\usepackage{subcaption}

\begin{document}
\title{Twist spinning of knots and metabolizers of Blanchfield pairings}
\author{Stefan Friedl}
\address{Fakult\"at f\"ur Mathematik\\ Universit\"at Regensburg\\   Germany}
\email{sfriedl@gmail.com}
\author{Patrick Orson}
\address{University of Edinburgh\\United Kingdom}
\email{patrickorson@gmail.com}
\date{\today}
\begin{abstract}
In a classic paper Zeeman introduced the $k$-twist spin of a knot $K$  and showed that the exterior of a twist spin fibers over $S^1$. In particular this result shows that the knot $K\#-K$ is doubly slice. In this paper we give a quick proof of Zeeman's result. The $k$-twist spin of $K$ also gives rise to two metabolizers for $K\#-K$  and we determine these two metabolizers precisely.

\end{abstract}
\maketitle

\section{Introduction}

Throughout this note we fix a category \textit{CAT} where  \textit{CAT}$=$\textit{DIFF}, \textit{PL} or \textit{TOP}. For $k\in \Z$ and a knot $K\subset S^{n+2}$, Zeeman \cite[p.~487]{Ze65} introduced the construction of a knot $S_{k}(K)\subset S^{n+3}$, called the $k$-twist spin of $K$.
We recall the definition in Section \ref{section:proofs}.
The following theorem is the main result in \cite{Ze65}.
 
 \begin{theorem}\label{mainthm}
If $k\ne 0$, then the $($closed$)$ knot exterior $S^{n+3}\sm \nu S_k(K)$ fibers over $S^1$, where the fiber is the result of removing an open ball from the $k$-fold branched cover of $K$.
\end{theorem}

A \emph{slice disk} for a knot $K\subset S^{n+2}$ is an embedded $(n+1)$-ball $B$ in $D^{n+3}$ such that $\partial B=K$. If $K$ admits a slice disk, then we say that 
$K$ is  \textit{slice}. Note that if $K$ is slice, then the double of the slice disk gives rise to an $(n+1)$-sphere in $S^{n+3}$ whose intersection with $S^{n+2}$ is precisely $K$. 

A knot $K\subset S^{n+2}$ is called \textit{doubly slice}
if $K$ is the intersection of  an unknot $U\subset S^{n+3}$ with the equator sphere $S^{n+2}\subset S^{n+3}$. By the above a knot which is doubly slice is also slice, but note that in general the converse does not hold.
We refer to \cite{Su71,St78,Ki06} for more details.  

It is well-known that for any $k$ the intersection $S_k(K)\cap S^{n+2}$ is isotopic to $K\#-K$. 
Furthermore, it is straightforward to see that the fiber in Theorem \ref{mainthm} for $k=\pm 1$ is a ball.
As was pointed out by Sumners \cite[Corollary~2.9]{Su71} Zeeman's theorem therefore has the following corollary.

\begin{corollary}\label{maincor}
For any knot $K\subset S^{n+2}$ the connect sum $K\#-K$ is doubly slice.
\end{corollary}

An unscientific poll among the authors and a wider group of topologists showed  that the statements of Theorem \ref{mainthm} and of Corollary \ref{maincor} are both well-known but that the proofs are less well understood. 
In this note we therefore present a short, self-contained proof of Theorem \ref{mainthm}.
Our approach is very  explicit, decomposing the exterior of the knot $S_k(K)$ into the union of two appropriately chosen components we simply write down the fibre bundle structures on each separately and then glue them together. This is not the first reproof of Theorem \ref{mainthm} and in \cite[Corollary 1.11]{GK78} a different technique is applied to recover the same result. A `disk knot version' of Corollary \ref{maincor} was also proved by Levine \cite[Theorem~C]{Lev83}.

We then turn our attention to the related algebra of Blanchfield forms.
Let $\Lambda=\Z[t^{\pm1}]$ and $\Omega=\Q(t)$. 
Given an odd-dimensional knot  $K\subset S^{2m+1}$ there exists a non-singular $(-1)^{m+1}$-hermitian pairing
\[ \lambda_K\colon H_m(S^{2m+1}\sm \nu K;\Lambda)\times H_m(S^{2m+1}\sm\nu K;\Lambda)\to \Omega/\Lambda,\]
known as the \emph{Blanchfield pairing}. 
We refer to \cite{Bl57} and \cite{Hi12} for details.
A \emph{metabolizer} for the Blanchfield pairing is a $\Lambda$-submodule $P\subset H:=H_m(S^{2m+1}\sm \nu K;\Lambda)$ such that 
\[ P=P^{\perp}:=\{v\in H\,|\, \lambda_K(v,w)=0\mbox{ for all }w\in P\}.\]

It is well-known that a slice disk gives rise to a metabolizer for the Blanchfield pairing of $K$. It also follows immediately from the definitions that $k$-twist spinning a knot gives rise to two slice disks for
$S_k(K)\cap S^{2m+1}$, which is isotopic to $K\#-K$. 

Our second main theorem determines the corresponding two metabolizers precisely. Here in the introduction we give a slightly informal statement.

\begin{theorem}\label{thm:metabsintro}
Let $K\subset S^{2m+1}$ be an oriented knot and let $k\in \Z$. We write $H=H_m(S^{2m+1}\sm \nu K;\Lambda)$.
 Then there exists an isomorphism 
$f\colon H\oplus H\to H_m(S^{2m+1}\sm\nu (K\#-K);\Lambda)$ which induces an isomorphism of Blanchfield forms
\[ \lambda_K\oplus -\lambda_K\to \lambda_{K\#-K}\]
such that  the two metabolizers corresponding to the two slice disks arising from twist spinning are 
\[ \{ v\oplus -v\,|\, v\in H\} \mbox{ and } \{ v\oplus -t^kv\,|\, v\in H\}.\]
\end{theorem}

We refer to Theorems \ref{thm:metabseven} and \ref{thm:metabsodd} for a much more precise, and consequently considerably longer, formulation. It was a surprise to the authors how difficult it was to make the statement of Theorem \ref{thm:metabsintro} rigorous and in Section \ref{section:metabs} we take great care to use precise arguments that keep track of the effects of changing base-points.

\subsection*{Acknowledgment.} We are grateful to the University of Glasgow for its hospitality and we wish to thank Danny Ruberman and Matthias Nagel for helpful comments. The second author wishes to thank his advisor Andrew Ranicki for his generous advice and patient guidance. We also wish to thank the referee for several helpful comments.

\section{Setup and Proof}\label{section:proofs}

Throughout this paper, given $k<l$ we view $S^k$ as the subset of $S^l\subset\R^{l+1}$ given by setting the first $l-k$ coordinates to zero.  
Given $k$ and $l$ we furthermore pick an identification of $D^k\times D^l$ with $D^{k+l}$. As usual we view $D^2$ also as a subset of $\C$. If $U$ is a submanifold of a manifold $V$ we use the notation $\nu U$ for an open tubular neighbourhood of $U$ in $V$.

Let $K\subset S^{n+2}$ be an oriented knot. We can  write $S^{n+2}=D^{n+2}\cup_{S^{n+1}}\tilde{D}^{n+2}$ as the union of two $(n+2)$-balls
in such a way that 
$K\cap \tilde{D}^{n+2}=0\times \tilde{D}^n$ is the trivial disk knot in $\tilde{D}^2\times \tilde{D}^n=\tilde{D}^{n+2}$.  We write $J:=K\cap D^{n+2}$, the other disk knot in the decomposition.

Given $z\in S^1$ we denote by 
\[ \begin{array}{rcl} \rho_z\colon D^{n+2}=D^2\times D^n&\to & D^{n+2}=D^2\times D^n\\
(w,x)&\mapsto & (zw,x)\end{array}\]
the rotation by $z$ in the $D^2$-factor. Note that $\rho_z$ restricts to the identity on $J\cap S^{n+1}$. 
Also note that we can and will assume that the decomposition $D^{n+2}=D^2\times D^n$ is oriented in such a way that for any $x\in S^{n+1}\sm \nu J$ the closed curve 
\[ \begin{array}{rcl} S^1&\to & D^2\times D^{n}\sm \nu J \\
z&\mapsto  & \rho_z(x)\end{array}\]
gives the oriented  meridian of $K$. 

Now let $k\in \Z$. In order to define the $k$-twist spin 
 of $K$, we use the following decomposition 
\[ S^{n+3}=S^1\times D^{n+2}\,\cup \, D^2\times S^{n+1}.\]
Denote by $\Phi_k$ the diffeomorphism
\[ \begin{array}{rcl} \Phi_k\colon S^1\times D^{n+2}&\to & S^1\times D^{n+2}\\
(z,x)&\mapsto & (z,\rho_{z^k}(x)).\end{array}\]
The $k$-twist spin $S_k(K)$ is then defined as 
\[ S_k(K):=
\underset{\subset S^1\times D^{n+2}}{\underbrace{ \Phi_k(S^1\times J)}}\,\,\cup\,\,
\underset{\subset D^2\times S^{n+1}}{\underbrace{ 
D^2\times  S^{n-1}}}.\]
More informally,  $S_k(K)$ is given by spinning the disk knot $J$ around the $S^1$-direction, performing $k$ twists around $J$ as you go, and then capping off the result by $D^2\times S^{n-1}$.

\begin{proof}[Proof of Theorem \ref{mainthm}]
The proof consists of two parts. We will first describe $S^{n+3}\sm \nu S_k(K)$ in a different, more convenient, way. We will then use this description to write down the promised fiber bundle over $S^1$. The first part is well-known, in fact this description of $S^{n+3}\sm \nu S_k(K)$ is also given in \cite[p.~201]{Fr05}.

We  write $Y:=D^{n+2}\sm  \nu J$.
Note that $Y\cap \partial D^{n+2}=S^{n+1}\sm \nu S^{n-1}$. 
As usual we can identify $S^{n+1}\sm \nu S^{n-1}$ with $S^1 \times D^n$.
Now we  see that 
\[ \begin{array}{rcl} S^{n+3}\sm \nu S_k(K)&=&S^1\times D^{n+2}\,\sm\, \Phi_k(S^1\times \nu J)\,\,\,\cup \,\,\,D^2\times  (S^{n+1}\sm \nu S^{n-1})\\
&=& \Phi_k(S^1\times Y)\,\,\,\cup \,\,\,D^2\times S^1 \times D^n.\end{array}\]
Note that  ${\Phi_k}$ restricts to an automorphism
of $S^1\times (Y\cap S^{n+1})=Y\cap \partial D^{n+2}=S^1\times S^1 \times D^n$.
We can thus glue $S^1\times Y$ and $D^2\times S^1 \times D^n$ together via the restriction of  ${\Phi_k}$ to $S^1\times S^1 \times D^n$.
The map
\[  S^1\times Y\,\,\cup_{{\Phi_k}} \,\,D^2\times S^1 \times D^n\to
\Phi_k(S^1\times Y)\,\,\cup \,\,D^2\times S^1 \times D^n,\]
which is given by $\Phi_k$ on the first subset and by the identity on the second subset, is then evidently a well-defined diffeomorphism.

We will  use this description of $S^{n+3}\sm \nu S_k(K)$ on the left to write down the fibre bundle structure over $S^1$. 
First, elementary obstruction theory shows that there exists a map $\varphi\colon Y\to S^1$ such that the restriction of $\varphi$ to $Y\cap \partial D^{n+2}=S^{n+1}\sm \nu S^{n-1}$ is just the projection map
 $S^{n+1}\sm \nu S^{n-1}=S^1 \times D^n\to S^1$.
 
\begin{claim}
The map
\begin{equation} \label{equ:fiber} \begin{array}{rcl} p\colon S^1\times Y&\to& S^1\\
(z,x)&\mapsto & z^{-k}\varphi( x)\end{array}\end{equation}
defines a fiber bundle with fiber
\[ \{ (z,x)\in S^1\times Y\,|\, z^{-k}\varphi(x)=1\}.\]
\end{claim}

Given $w\in S^1$ we have 
\[ p^{-1}(w)= \{ (z,x)\in S^1\times Y\,|\, z^{-k}\varphi(x)=w\}.\]
Let $w,w'\in S^1$. We pick a $k$-th root $\xi$ of $w^{-1}w'$. Then the map $(z,x)\to (z\xi,x)$ defines a homeomorphism $p^{-1}(w)\to p^{-1}(w')$. It is now straightforward to see that $p$ is in fact a fiber bundle.
This concludes the proof of the claim.

It is straightforward to verify that the assumption that $k\ne 0$ implies that the map
\begin{equation} \label{equ:fiber} \begin{array}{rcl} p\colon S^1\times Y&\to& S^1\\
(z,x)&\mapsto & z^{-k}\varphi( x)\end{array}\end{equation}
defines a fiber bundle.

It  follows from the definitions that the map 
\[  S^1\times Y\,\,\cup_{{\Phi_k}} \,\,D^2\times S^1 \times D^n\to S^1\]
which is given by $p$ on the first subset and by projection on the $S^1$-factor in the second subset is the projection of a fiber bundle. 

It remains to identify the fiber of the fibration.
The fiber `on the right' (of the decomposition) is $D^{2}\times \{1\}\times D^n$ whereas the fiber `on the left' is given by 
\[ Y_k=\{ (z,x)\in S^1\times Y\,|\, \varphi(x)=z^{k}\}\]
which is just the $k$-fold cyclic cover of $Y$ corresponding to the epimorphism
$\pi_1(Y)\to H_1(Y;\Z)\xrightarrow{\cong} \Z\to \Z/k$. 

Note that $Y$ is in fact diffeomorphic to the knot exterior $S^{n+2}\sm \nu K$, and that hence $Y_k$ is just the $k$-fold cyclic cover of $S^{n+2}\sm \nu K$. 
It is  straightforward
 to see that the fiber
\[ Y_k\cup_{S^1\times \{1\}\times D^n} D^{2}\times \{1\}\times D^n\]
is the result of attaching a 2-handle to $Y_k$ along the preimage of a meridian under the covering map $Y_k\to Y$. Put differently, the fiber is obtained 
by removing an open ball from the $k$-fold branched cover of $K$.
\end{proof}

We immediately obtain the following corollary.

\begin{corollary}\label{cor:sk1trivial}
If $K\subset S^{n+2}$ is a knot, then $S_{\pm 1}(K)\subset S^{n+3}$ is a trivial knot.
\end{corollary}

\begin{proof}
First note that  the $\pm 1$-fold branched cover of $S^{n+2}$ along $K$ is just $S^{n+2}$ again. It thus follows  from Theorem \ref{mainthm} that $S_{\pm 1}(K)$ bounds an $(n+2)$-ball in $S^{n+3}$, which means that  $S_{\pm 1}(K)\subset S^{n+3}$ is a trivial knot.
\end{proof}

We also make following observation concerning twist spins.

\begin{lemma}\label{lem:kplus-k}
If $K\subset S^{n+2}$ is a knot, then for any $k\in\Z$ the knot $ S_{k}(K)\cap S^{n+2}$ is isotopic to $K\# -K$.
\end{lemma}

\begin{proof}
We use the notation in the definition of the twist spins of $K$. 
In particular we denote by $J\subset D^{n+2}$ the disk knot
corresponding  to $K$. We denote by $J'$ the string knot which is defined by $\Phi(-1\times J)=-1\times J'$. 
Put differently, $J'$ is  the result of rotating $J\subset D^2\times D^n=D^{n+2}$  by $k\pi$.  Note that $J'$ is isotopic in $D^{n+2}$ to $J$ rel the boundary.
We write 
\[ S^{n+3}=S^1\times  D^{n+2}\,\,\cup \,\, D^2\times S^{n+1}\]
with equator sphere
\[ S^{n+2}=\{ \pm 1\} \times D^{n+2} \,\,\cup \,\, D^1\times S^{n+1}.\]
The above decomposition of $S^{n+3}$ gives rise to an  orientation preserving map 
\[ \Psi\colon S^1\times D^{n+2} \to S^{n+3}\]
such that 
\[ \Psi(S^1\times D^{n+2})\cap S^{n+2} = \{-1\} \times D^{n+2}\,\,\cup \{1\}\times D^{n+2}.\]
Note that the restriction of $\Psi$ to $ \{-1\} \times D^{n+2}$
is \textit{orientation reversing} and that the restriction of $\Psi$ to $ \{1\} \times D^{n+2}$
is \textit{orientation preserving}. In particular
$\Phi_k(S^1\times J)\cap S^{n+2}$ is the union of $J$ with the mirror image of $J'$. 

Since $J$ and $J'$ are isotopic rel the boundary it follows easily that  $S_k(K)\cap S^{n+2}$ is isotopic to the connected sum of $K$ and $-K$.
\end{proof} 

We finally recall that a knot $K\subset S^{n+2}$ is called \textit{doubly slice}
if there exists an unknot $U\subset S^{n+3}$ with $U\cap S^{n+2}=K$.
The following corollary is  an immediate consequence of 
Corollary \ref{cor:sk1trivial} and Lemma \ref{lem:kplus-k}. This consequence was first observed by Sumners \cite[Corollary~2.9]{Su71}.

\begin{corollary}
If $K\subset S^{n+2}$ is an oriented knot, then $K\#-K$ is doubly slice.
\end{corollary}

\section{Base points and infinite cyclic covers}\label{section:modules}

In this section we will quickly bring into focus several indeterminacy issues for infinite cyclic covers which often get swept under the carpet.

Let $X$ be a connected topological space with $H_1(X)\cong \Z$, equipped
with  an identification $H_1(X)=\Z$.  We pick a base point $x\in X$. 
We denote by $\wti{X}_{x}\to X$ the infinite cyclic cover corresponding to the canonical epimorphism \[ \phi_{x}\colon \pi_1(X,x)\to H_1(X)=\Z=\ll t\rr.\] 
Note that $\wti{X}_{x}$ has a canonical action by the deck transformation group $\Z=\ll t\rr$.
In particular we can view $H_{i}(\wti{X}_{x})$ as  a module over the group ring of $\Z=\ll t\rr$, i.e. over $\Lambda=\zt$. We henceforth write
\[ H_{i}^{x}(X;\Lambda):=H_{i}(\wti{X}_{x}).\]

The question now arises, whether these homology $\Lambda$-modules depend on the choice of the base point $x$.
If $y$ is a different base point, then we can pick a path $p$ from $x$ to $y$
which then defines an isomorphism $p_*\colon H_{i}^{x}(X;\Lambda)\to H_{i}^{y}(X;\Lambda)$.
We thus see that the isomorphism type of the homology $\Lambda$-modules does not depend on the choice of the base point.
In the following we denote by $H_i(X;\Lambda)$ the isomorphism type of the $\Lambda$-module.

The next question which arises is, to what degree does the isomorphism $p_*$ depend on the choice of the path $p$.
If $q$ is another path from $x$ to $y$, then it is straightforward to see that 
\[ q_*^{-1}\circ p_*\colon H_{i}^{x}(X;\Lambda)\to H_{i}^{x}(X;\Lambda)\]
is multiplication by $t^{\phi_x(\ol{q}p)}$, where $\ol{q}$ is the same path as $q$ but with opposite orientation.

Now let $Y$ be a connected subspace of $X$ which contains the base point $x$ and such that the inclusion induces an isomorphism $H_1(Y;\Z)\cong H_1(X;\Z)$.
We then obtain an induced map of infinite cyclic covers $\wti{Y}_x\to \wti{X}_x$, in particular we obtain for each $i$ an induced map
\[ H_i^x(Y;\Lambda)\to H_i^x(X;\Lambda).\] 

Now we turn to the study of infinite cyclic covers of knots and disk knots. 
Let $K\subset S^{n+2}$ be an oriented knot. 
Note that $H_1(S^{n+2}\sm \nu K;\Z)\cong \Z$ and we identify $H_1(S^{n+2}\sm \nu K;\Z)$ with $\Z$ by identifying the oriented meridian of $K$ with $1$. 

Suppose we are given a decomposition $S^{n+2}=D^{n+2}\cup_{S^{n+1}}\tilde{D}^{n+2}$ as the union of two $(n+2)$-balls
in such a way that 
$K\cap \tilde{D}^{n+2}=0\times \tilde{D}^n$ is the trivial disk knot in $\tilde{D}^2\times \tilde{D}^n=\tilde{D}^{n+2}$.  We write $J:=K\cap D^{n+2}$.
Note that the inclusion induces an isomorphism $H_1(D^{n+2}\sm \nu J;\Z)\to H_1(S^{n+2}\sm \nu K;\Z)$. We use this isomorphism to identify $H_1(D^{n+2}\sm \nu J;\Z)$ with $\Z$. Now we pick a base point $x\in (D^{n+2}\sm \nu J)\cap  S^{n+1}$.
A straightforward Mayer--Vietoris argument shows that for any $i$ the inclusion induces  an isomorphism
\[ H_i^x(D^{n+2}\sm \nu J;\Lambda)\xrightarrow{\cong }H_i^x(S^{n+2}\sm \nu K;\Lambda).\]

\section{Metabolizers for Blanchfield pairings}\label{section:metabs}

Throughout this section we write $\Lambda:=\zt$ and $\Omega:=\Q(t)$. We view these as rings with involution given by $t\mapsto t^{-1}$ extended trivially linearly to the coefficients. 
Throughout this section let $K\subset S^{2m+1}$ be an odd-dimensional knot. As we mentioned in the introduction,  there exists a non-singular $(-1)^{m+1}$-hermitian pairing
\[ \lambda_K\colon H_m(S^{2m+1}\sm \nu K;\Lambda)\times H_m(S^{2m+1}\sm \nu K;\Lambda)\to \Omega/\Lambda,\]
known as the \emph{Blanchfield pairing}. 
We refer to \cite{Bl57} and \cite{Hi12} for details. Recall that a \emph{metabolizer} for the Blanchfield pairing is a $\Lambda$-submodule $P\subset H:=H_m(S^{2m+1}\sm \nu K;\Lambda)$ such that 
\[ P=P^{\perp}:=\{v\in H\,|\, \lambda_K(v,w)=0\mbox{ for all }w\in P\}.\]

Recall that a \emph{slice disk} for the knot $K\subset S^{2m+1}$ is an embedded $2m$-ball $B$ in $D^{2m+2}$ such that $\partial B=K$. A well-known Poincar\'e duality argument shows that the inclusion induced map 
$H_1(S^{2m+1}\sm \nu K;\Z)\to H_1(D^{2m+2}\sm \nu B;\Z)$ is an isomorphism.
Given a base point $x\in S^{2m+1}\sm \nu K$ we can therefore in particular consider the induced map
\[ H_m^x(S^{2m+1}\sm \nu K;\Lambda)\to H_m^x(D^{2m+2}\sm \nu B;\Lambda).\]

The following proposition shows that  a slice disk gives rise to a metabolizer for the Blanchfield pairing of $K$.
We refer to \cite{Ke75} and \cite[Proposition~2.8]{Let00} for the proof. 

\begin{theorem}\label{thm:metab}
Let $B\subset D^{2m+2}$ be a slice disk for $K\subset S^{2m+1}$. Then 
\[ \ker\left(H_m(S^{2m+1}\sm \nu K;\Lambda)\to H_m(D^{2m+2}\sm \nu B;\,\Lambda)\,/\,\mbox{$\Z$-torsion}\right)
\]
is a metabolizer for the Blanchfield pairing $\lambda_K$ of $K$.
\end{theorem}

Henceforth we consider  the $k$-twist spin $S_k(K)$ of the knot $K\subset S^{2m+1}$.
Recall  that we can then write $S^{2m+1}=D^{2m+1}\cup_{S^{2m}}\tilde{D}^{2m+1}$ as the union of two $(2m+1)$-balls
in such a way that 
$K\cap \tilde{D}^{2m+1}=0\times \tilde{D}^{2m-1}$ is the trivial disk knot in $\tilde{D}^2\times \tilde{D}^{2m-1}=\tilde{D}^{2m+1}$.  We write $J:=K\cap D^{2m+1}$, the other disk knot in the decomposition.
Recall that the $k$-twist spin $S_k(K)$ of the knot $K\subset S^{2m+1}$ is then defined as 
\[ S_k(K):=
\underset{\subset S^1\times D^{2m+1}}{\underbrace{ \Phi_k(S^1\times J)}}\,\,\cup\,\,
\underset{\subset D^2\times S^{2m}}{\underbrace{ 
D^2\times  S^{2m-2}}}.\]
Now we write
\[ \begin{array}{rcl} S_+^1&=&\{ z\in S^1\,|\, \im(z)\geq 0\}, \\
 D_+^{2m+2}&=& S^1_+\times D^{2m+1}\cup \{ z\in D^2\,|\, \im(z)\geq 0\}\times S^{2m}.\end{array}\]
We similarly define $S_-^1$ and $D_-^{2m+2}$. Note that $D_-^{2m+2}\cup D_+^{2m+2}=S^{2m+2}$
and that  $D_-^{2m+2}\cap D_+^{2m+2}=S^{2m+1}$.
Also note that $B_+:=S_k(K)\cap D_+^{2m+2}$ and $B_-:=S_k(K)\cap D_-^{2m+2}$ are slice disks for 
\[ L:= S_{k}(K)\cap S^{2m+1}.\]
As we have seen in Theorem \ref{thm:metab}, the slice disks $B_+$ and $B_-$ give rise to metabolizers for the Blanchfield pairing  of $L=S_k(K)\cap S^{2m+1}$. We start out with the following lemma.

\begin{lemma}\label{lem:torsion-free}
Let $K\subset S^{2m+1}$ be an oriented knot. We write $H=H_m(S^{2m+1}\sm  K;\Lambda)$.
We furthermore write $L=S_{k}(K)\cap S^{2m+1}$,  $B_+:=S_k(K)\cap D_+^{2m+2}$ and $B_-:=S_k(K)\cap D_-^{2m+2}$. Then the modules  $H_m(D^{2m+2}_-\sm \nu B_-;\Lambda)$ and $H_m(D^{2m+2}_+\sm \nu B_+;\Lambda)$ are $\Z$-torsion free.
\end{lemma} 

\begin{proof}
It follows easily from the definitions that the inclusion induced maps
\[ S^{2m+1}\sm \nu K\leftarrow D^{2m+1}\sm \nu J \to (D^{2m+2}_\pm\sm \nu B_\pm)\cap S_\pm^1\times D^{2m+1}\to D^{2m+2}_\pm\sm \nu B_\pm\]
 are homotopy equivalences. It follows that 
 the modules  $H_m(D^{2m+2}_-\sm \nu B_\pm;\Lambda)$ are isomorphic to 
 $H_m(S^{2m+1}\sm \nu K;\Lambda)$ of $K$, which is well-known (see \cite{Lev77}) to be  $\Z$-torsion free.
\end{proof}

The combination of Theorem \ref{thm:metab} with Lemma  \ref{lem:torsion-free} shows that 
\[ P_{\pm}:=\ker\left\{ H_m(S^{2m+1}\sm \nu L;\Lambda) \to  H_m(D^{2m+2}_\pm\sm \nu  B_\pm;\Lambda)\right\}\]
are metabolizers of the Blanchfield form of $L=S_k(K)\cap S^{2m+1}$.
We write $H=H_m(S^{2m+1}\sm \nu K;\Lambda)$. Recall that $L$ is isotopic to $K\#-K$. It is well-known that 
there exists an isomorphism
\begin{equation} \label{equ:iso} H\oplus H\xrightarrow{\cong}H_m(S^{2m+1}\sm \nu (K\#-K);\Lambda)\xrightarrow{\cong} H_m(S^{2m+1}\sm \nu L;\Lambda).\end{equation}
It is therefore tempting to write down $P_{\pm}$ as submodules of $H\oplus H$. 
But this undertaking is fraught with difficulties since the isomorphism in (\ref{equ:iso}) is not canonical and depends on various choices of base points and connecting paths. In the following we will carefully pick an isomorphism as in (\ref{equ:iso}) and then describe the submodules of $H\oplus H$ corresponding to $P_\pm$.

The discussion now naturally breaks up into two cases, either $k$ is even, in which case $S_k(K)\cap (-1\times D^{2m+1})=-1\times J$, or $k$ is odd, in which case $S_k(K)\cap (-1\times D^{2m+1})=-1\times \rho_{-1}(J)$.
The two subsequent  Theorems \ref{thm:metabseven} and \ref{thm:metabsodd}
are  the promised more precise version of Theorem \ref{thm:metabsintro}.

\begin{theorem}\label{thm:metabseven}
Let $K\subset S^{2m+1}$ be an oriented knot and let $k\in \Z$ be even. We define $J,L,B_+$ and $B_-$ as above. Let $x\in (D^{2m+1}\sm \nu J)\cap S^{2m+1}$ be a base point. 
We write $H=H_m^x(S^{2m+1}\sm \nu K;\Lambda)$. We  denote by $\Phi$ the map
\[  H\xleftarrow{\cong} H_m^x(D^{2m+1}\sm \nu J;\Lambda)
\xrightarrow{\cong} H_m^{1\times x}(1\times (D^{2m+1}\sm \nu J);\Lambda)\to H_m^{1\times x}(S^{2m+1}\sm \nu L;\Lambda),\]
where the left and right maps are induced by inclusions and where the middle map is the obvious isomorphism. 
 We pick a path $\g$ in $D^1\times S^{2m}\subset S^{2m+1}=\pm 1\times D^{2m+1}\cup D^1\times S^{2m}$ from $-1\times x$ to $1\times x$. We denote by $\Psi$ the map
\[ \begin{array}{rcl} H\xleftarrow{\cong} H_m^x(D^{2m+1}\sm \nu J;\Lambda)&\xrightarrow{\cong} &
H_m^{-1\times x}(-1\times (D^{2m+1}\sm \nu J);\Lambda)\\
&\to& H_m^{-1\times x}(S^{2m+1}\sm \nu L;\Lambda) \xrightarrow{\g_*} H_m^{1\times x}(S^{2m+1}\sm \nu L;\Lambda),\end{array}\]
where the first and the third map are induced by inclusions, the second map is the obvious isomorphism
and the fourth map is induced by the change of base point using the path $\g$.
 Then $\Phi\oplus \Psi$ induces an isomorphism of pairings
\[ \lambda_K\oplus -\lambda_K\to \lambda_L\]
such that the two metabolizers arising from twist spinning
\[\begin{array}{rcl}&& \ker\left( H\oplus H\xrightarrow{\Phi\oplus \Psi} H_m^{1\times x}(S^{2m+1}\sm \nu L;\Lambda)\to H_m^{1\times x}(D^{2m+2}_-\sm \nu  B_-;\Lambda)\right)\\[2mm]
\text{and}&\,& \ker\left( H\oplus H\xrightarrow{\Phi\oplus \Psi} H_m^{1\times x}(S^{2m+1}\sm \nu L;\Lambda)\to H_m^{1\times x}(D^{2m+2}_+\sm \nu  B_+;\Lambda)\right)\end{array}\]are respectively equal to
\[\{ v\oplus -t^{\frac{k}{2}}v\,|\, v\in H\} \quad\text{and}\quad\{ t^{\frac{k}{2}}v\oplus -v\,|\, v\in H\}.\]
\end{theorem} 

\begin{proof}
We write $X:=S^{2m+1}\sm \nu K$ and we denote by $c\colon \what{X}\to X$ the infinite cyclic covering
of $X$ corresponding to the base point $x$ and corresponding to the kernel of the epimorphism
$\pi_1(X,x)\to H_1(X;\Z)\xrightarrow{\cong} \langle t\rangle$ which sends an oriented meridian of $K$ to $t$.
Given any subset $U$ of $X$ we henceforth write $\what{U}=c^{-1}(U)$.

Now we write $Y:=D^{2m+1}\sm \nu J$. As discussed in Section \ref{section:modules} the inclusion induces an
isomorphism
\[ H_m(\what{Y})\to H_m(\what{X})=H_m^x(X;\Lambda)=:H\]
This allows us to make the identification $H=H_m(\what{Y})$.

We write $W:=S^{2m+2}\sm S_k(K)$ where we again decompose $S^{2m+2}$ as $S^1\times D^{2m+1}\cup D^2\times S^{2m}$.
We equip $W$ with the base point $1\times x\in S^1\times D^{2m+1}$. We denote by $c\colon \wti{W}\to W$ the infinite cyclic covering
of $W$ corresponding to t the epimorphism
$\pi_1(W,1\times x)\to H_1(W;\Z)\xrightarrow{\cong} \langle t\rangle$ which sends an oriented meridian of $K$ in $D^{2m+1}$ to $t$.
Throughout the proof we think of $\wti{W}$ as equivalence classes of paths emanating from the base point $1\times x$.
As before, given any subset $U$ of $W$ we henceforth write $\wti{U}=c^{-1}(U)$.
Note that with our conventions we have a canonical homeomorphism $\what{Y}\to \wti{Y}$. 
We will henceforth make the identification $H=H_m(\wti{Y})$. 

We write $W_0:=W\cap S^1\times D^{2m+1}$ and consider the map
\[ q\colon \wti{W}_0\to W_0\to S^1\times D^{2m+1}\to S^1\]
where the last map is just projection onto the first factor. 
Note that $q^{-1}(1)=\wti{Y}$. We refer to the figure below for an illustration.

Now we consider the homotopy
\[ \begin{array}{rcl} h_s\colon W_0&\to &W_0 \\
(z,p)&\mapsto &(e^{ is}z,\rho_{e^{ isk}}(p))\end{array}\]
with parameter $s\in \R$. Note that this homotopy lifts to a homotopy
\[ \wti{h}_s\colon \wti{W}_0\to \wti{W}_0 \]
with parameter $s\in \R$. In fact this lifting can be described very explicitly: given a  path $\alpha$ from the base point $1\times x$ to a point $(z,p)$ in $W_0$ we consider the path
\[ \begin{array}{rcl}\beta\colon [0,s] &\mapsto & W_0 \\
r&\mapsto& (ze^{ ir},\rho_{e^{ ir}}(p)).\end{array}\]
We then have $\wti{h}_s([\alpha])=[\beta\alpha]$. 
Note that for any $r,s\in \R$ we have $\wti{h_{r+s}}=\wti{h_r}\circ \wti{h_s}$. 

\begin{figure}[h]
\labellist
\small\hair 2pt
\pinlabel $\widetilde{W_0}$ [l] at 93 141
\pinlabel \textcolor{red}{$\widetilde{D^{2m+1}\sm \nu J}=q^{-1}(1)$} [l] at 117 128
\pinlabel $q$ [l] at 147 103
\pinlabel $S^1$ [l] at 271 119
\pinlabel \textcolor{blue}{$\widetilde{h_s}$} [l] at -13 63
\pinlabel $c$ [l] at 88 53
\pinlabel $W_0$ [l] at 208 41
\pinlabel \textcolor{blue}{$h_s$} [l] at 106 4
\pinlabel \textcolor{red}{$D^{2m+1}\sm \nu J$} [l] at 198 6
\endlabellist
\begin{center}
\includegraphics{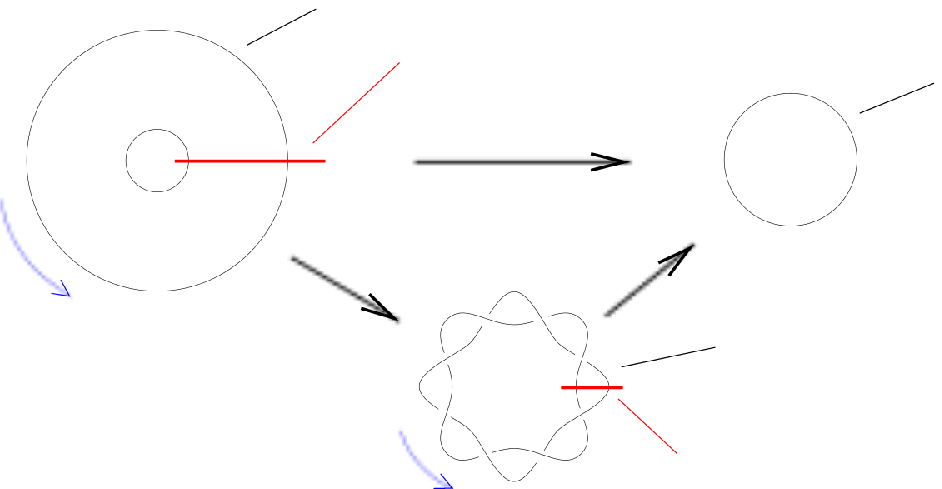}
\end{center}
\label{fig:maps}
\end{figure}

Also note that for any $z\in S^1$ and $s\in \R$ the map $h_{s}$ induces an isomorphism
\[ {\wti{h}_{s}}\colon H_m(q^{-1}(z))\to H_m(q^{-1}(e^{ is}z)).\]
Furthermore, for any interval $I$ in $S^1$ with end points $z,z'=ze^{ is}$, $s\in (0,2\pi)$,  the inclusion maps induce  isomorphisms
\[ \i_z\colon H_m(q^{-1}(z))\to H_m(q^{-1}(I)) \mbox{ and }
\i_{z'}\colon H_m(q^{-1}(z'))\to H_m(q^{-1}(I))\]
such that 
\begin{equation} \label{equ:iab} \i_{z'}^{-1}\circ \i_z={\wti{h}_{s}}.\end{equation}
We can now formulate the following claim.

\begin{claim}
We denote by $f$ the map
\[ H=H_m(q^{-1}(1))\to  H_m(\wti{S^{2m+1}\sm \nu L})\]
and we denote by $g$ the map
\[  H=H_m(q^{-1}(1))\xrightarrow{\wti{h_\pi}} H_m(q^{-1}(-1))\to H_m(\wti{S^{2m+1}\sm \nu L}.)\]
Then $f\oplus g$ induces an isomorphism of pairings
\[ \l_K\oplus -\l_K\xrightarrow{\cong} \l_L\]
such that
\[\begin{array}{rcl}
\ker\left(H\oplus H\xrightarrow{f\oplus g} H_m(S^{2m+1}\sm \nu L;\Lambda)\to H_m(D^{2m+2}_+\sm \nu B_+;\Lambda)\right)
&=&\{ v\oplus -v\,|\, v\in H\}, \mbox{ and }\\[2mm]
\ker\left(H\oplus H\xrightarrow{f\oplus g} H_m(S^{2m+1}\sm \nu L;\Lambda)\to H_m(D^{2m+2}_-\sm \nu B_-;\Lambda)\right)
&=&\{ v\oplus -t^kv\,|\, v\in H\}.\end{array}\]
\end{claim} 
\medskip

In order to prove the claim we  first consider the following commutative diagrams of inclusion induced maps:
\[ \xymatrix{ H_m(q^{-1}(1))\oplus  H_m(q^{-1}(-1))\ar[r]\ar[d] & H_m(q^{-1}(S_{+}^1))=H_m(\wti{W_0\cap D_+^{2m+2}})\ar[d]\\
H_m(\wti{S^{2m+1}\sm \nu L}) \ar[r] & H_m(\wti{D^{2m+2}\sm \nu B_+})=H_m(\wti{W\cap D_+^{2m+2}}).}\]
Using a Mayer-Vietoris argument it is straightforward to see that the vertical  maps are isomorphisms.
It follows from the above commutative diagram and from the definitions that 
\[ 
\begin{array}{rcl}
&&\ker\left(H\oplus H\xrightarrow{f\oplus g} H_m(S^{2m+1}\sm \nu L;\Lambda)\to H_m(D^{2m+2}_+\sm \nu B_+;\Lambda)\right)\\[2mm]
&=&\ker\left(\i_1\oplus (\i_{-1}\circ \wti{h_{\pi}})\colon H_m(q^{-1}(1))\oplus H_m(q^{-1}(1))\to H_m(q^{-1}(S_+^1))\right).\end{array}\]
By (\ref{equ:iab}) we have $\wti{h}_\pi=\i_{-1}^{-1}\circ \i_1$. It thus follows that the above kernel equals 
\[ \ker\left( \i_1\oplus \i_1\colon H\oplus H\to H_m(q^{-1}(S_+^1))\right)=\{ v\oplus -v\,|\, v\in H\}.\]

Essentially the same argument shows that 
\[ 
\begin{array}{rcl}
&&\ker\left(H\oplus H\xrightarrow{f\oplus g} H_m(S^{2m+1}\sm \nu L;\Lambda)\to H_m(D^{2m+2}_-\sm \nu B_-;\Lambda)\right)\\[2mm]
&=&\ker\left(\i_1\oplus (\i_{-1}\circ \wti{h_{\pi}})\colon H_m(q^{-1}(1))\oplus H_m(q^{-1}(1))\to H_m(q^{-1}(S_-^1))\right).\end{array}\]
By (\ref{equ:iab}) we have $\wti{h}_\pi=\i_{1}^{-1}\circ \i_{-1}$, i.e.
$\i_{-1}=\i_1\circ \wti{h}_\pi$. It thus follows that the above kernel equals 
\[ \ker\left( \i_1\oplus (\i_1\circ \wti{h_{2\pi}})\colon H\oplus H\to H_m(q^{-1}(S_-^1))\right).\]
On the other hand it follows from (\ref{equ:fiber}) that the map 
\[ \wti{h_{2\pi}}\colon H=H_m(q^{-1}(1))\to H=H_m(q^{-1}(1))\]
is multiplication by $t^{-k}$. 
It follows that 
\[ \ker\left( \i_1\oplus (\i_1\circ \wti{h_{2\pi}})\colon H\oplus H\to H_m(q^{-1}(S_-^1))\right)=\{ v\oplus -t^{k}v\,|\, v\in H\}.\]

It remains to show that 
 $f\oplus g$ induces an isomorphism of pairings
\[ \lambda_K\oplus -\lambda_K\to \lambda_L.\]
We write $J'=\rho_{-1}(J)\subset D^n$. Note that  $h_{\pi}$ induces an 
\emph{isotopy} from the disk knot
$J\subset D^{2m+1}$ to the disk knot $J'\subset D^{2m+1}$. 
With the sign conventions, see the proof of Lemma \ref{lem:kplus-k}, 
it  follows easily
that for any $v,w\in H_m(D^{2m+1} \sm \nu J;\Lambda)$ we have
$\lambda_{J'}(h_{\pi}(v),h_{\pi}(w))=-\lambda_{J}(v,w)$. 
This concludes the proof of the claim.\medskip

The theorem now follows from the final claim:

\begin{claim} 
For any $v\in H$ we have
\[ \Phi(v)=f(v) \mbox{ and } \Psi(v)=t^{-\frac{k}{2}}g(v).\]
\end{claim} 

To prove this, note that $f$, $g$, $\Phi$, and $\Psi$ are induced by chain maps which, in an abuse of notation, we denote by the same letters. Then any $v\in H=H_m(\wti{Y})$ can be represented by a finite sum of based $m$-chains $\sigma$ in $Y$, so it is sufficient to show the claim for such a $\sigma$. It follows from the definitions that $\Phi(\sigma)=f(\sigma)$. On the other hand, if we consider $\Psi(\sigma)$ and $g(\sigma)$, then we see that they are represented by the same chains but the basing is different. To make precise this difference, consider the following path in $W$
\[ \begin{array}{rcl} \alpha\colon [0,1]&\to & W\\
s&\mapsto & (e^{\pi is},\rho_{e^{\pi iks}}(x)).\end{array}\]
Then the basings for $\Psi(\sigma)$ and $g(\sigma)$ differ by the concatenation
of $\alpha$ and $\g$. But under the map $H_1(W;\Z)\to \ll t\rr$ the image of $[\alpha\g]$ is precisely $t^{\frac{k}{2}}$.
It follows that $g(\sigma)=t^{\frac{k}{2}}\Psi(\sigma)$.
This concludes the proof of the claim. 
\end{proof}

The case that $k$ is odd is a little more complicated since in this case $J'=\rho_{-1}(J)$.
For $\eps\in \{-1,1\}$ we consider the path
\[ \begin{array}{rcl} \delta^{\eps}\colon [0,1] &\to & -1\times D^{2m+1}\\
t&\mapsto &-1\times \rho_{e^{\eps\pi i(1-t)}}(x),\end{array}\]
Now we have the following theorem.

\begin{theorem}\label{thm:metabsodd}
Let $K\subset S^{2m+1}$ be an oriented knot and let $k\in \Z$ be odd. We define $J,L,B_+$ and $B_-$ as above. Let $x\in (D^{2m+1}\sm \nu J)\cap S^{2m+1}$ be a base point. 
We write $H=H_m^x(S^{2m+1}\sm \nu K;\Lambda)$. We  denote by $\Phi$ the map
\[  H\xleftarrow{\cong} H_m^x(D^{2m+1}\sm \nu J;\Lambda)
\xrightarrow{\cong} H_m^{1\times x}(1\times (D^{2m+1}\sm \nu J);\Lambda)\to H_m^{1\times x}(S^{2m+1}\sm \nu L;\Lambda),\]
where the left and right maps are induced by inclusions and where the middle map is the obvious isomorphism. 
Now we pick $\eps\in \{-1,1\}$ and we pick a path $\g$ in $D^1\times S^{2m}\subset S^{2m+1}=\pm 1\times D^{2m+1}\cup D^1\times S^{2m}$ from $-1\times x$ to $1\times x$. We denote by $\Psi$ the map
\[ \begin{array}{rcl} H\xleftarrow{\cong} H_m^x(D^{2m+1}\sm \nu J;\Lambda)&\xrightarrow{\cong} &
H_m^{-1\times x}(-1\times (D^{2m+1}\sm \nu J);\Lambda)\\
&\xrightarrow{\cong} & H_m^{-1\times \rho_{-1}(x)}(-1\times (D^{2m+1}\sm \nu J');\Lambda)\\
&\xrightarrow{\delta^{\eps}_*} & H_m^{-1\times x}(-1\times (D^{2m+1}\sm \nu J');\Lambda)\\
&\to& H_m^{-1\times x}(S^{2m+1}\sm \nu L;\Lambda) \xrightarrow{\g_*} H_m^{1\times x}(S^{2m+1}\sm \nu L;\Lambda),\end{array}\]
where the first and the fifth map are induced by inclusions, the second map is induced by the obvious homeomorphism,
the third map is induced by the homeomorphism  $\rho_{-1}$, the fourth map is induced by the path $\delta^\eps$, 
 and the last map is induced by the change of base point using the path $\g$.
 Then $\Phi\oplus \Psi$ induces an isomorphism of pairings
\[ \lambda_K\oplus -\lambda_K\to \lambda_L\]
such that the two metabolizers arising from twist spinning
\[\begin{array}{rcl}&& \ker\left( H\oplus H\xrightarrow{\Phi\oplus \Psi} H_m^{1\times x}(S^{2m+1}\sm \nu L;\Lambda)\to H_m^{1\times x}(D^{2m+2}_-\sm \nu  B_-;\Lambda)\right)\\[2mm]
\text{and}&\,& \ker\left( H\oplus H\xrightarrow{\Phi\oplus \Psi} H_m^{1\times x}(S^{2m+1}\sm \nu L;\Lambda)\to H_m^{1\times x}(D^{2m+2}_+\sm \nu  B_+;\Lambda)\right)\end{array}\]are respectively equal to
\[\{ v\oplus -t^{\frac{k+\eps}{2}}v\,|\, v\in H\} \quad\text{and}\quad\{ t^{\frac{k-\eps}{2}}v\oplus -v\,|\, v\in H\}.\]
\end{theorem} 
\medskip 

The proof is similar to the proof of Theorem \ref{thm:metabsodd} and we leave it as a refreshing exercise to the reader.


\begin{thebibliography}{10}
\bibitem[Bl57]{Bl57}
R. C. Blanchfield, {\em Intersection theory of manifolds with operators with applications to knot theory}, Ann. of Math. (2) 65 (1957), 340--356.
\bibitem[Fr05]{Fr05}
G. Friedman, {\em Knot spinning},  Handbook of knot theory, 187--208, Elsevier B. V., Amsterdam, 2005. 
 \bibitem[GK78]{GK78}
 D. Goldsmith, and L. H. Kauffman, {\em Twist spinning revisited}, Trans. Am. Math. Soc. 239 (1978), 229--251.
\bibitem[Hi12]{Hi12}
J. Hillman, {\em Algebraic invariants of links}, Second edition. Series on Knots and Everything, 52. World Scientific Publishing Co. (2012)
 \bibitem[Ke75]{Ke75}
C. Kearton, {\em Cobordism of knots and Blanchfield duality}, J. London Math.
Soc. (2) 10, no. 4 (1975), 406--408.
\bibitem[Ki06]{Ki06}
T. Kim, {\em New obstructions to doubly slicing knots},
Topology 45, No. 3 (2006), 543--566.
\bibitem[Let00]{Let00}
C. Letsche, {\em An obstruction to slicing knots using the eta invariant}, Math.
Proc. Cambridge Phil. Soc. 128, no. 2 (2000), 301--319.
\bibitem[Lev77]{Lev77}
J. Levine, {\em Knot modules. I.},
Trans. Am. Math. Soc. 229 (1977), 1--50.
 \bibitem[Lev83]{Lev83}
 J. Levine, {\em Doubly sliced knots and doubled disk knots}, Michigan Math. J. 30 (1983), 249--256.
\bibitem[St78]{St78}
 N. W. Stoltzfus, {\em Algebraic computations of the integral concordance and double null concordance group of knots},
Knot theory (Proc. Sem., Plans-sur-Bex, 1977), 274–290, Lecture Notes in Math., 685, Springer, Berlin, 1978. 
\bibitem[Su71]{Su71}
D. W. Sumners, {\em Invertible knot cobordisms}, Comment. Math. Helv. 46 (1971), 240--256. 
\bibitem[Ze65]{Ze65}
 E. C. Zeeman, {\em Twisting spun knots}, Trans. Am. Math. Soc. 115 (1965), 471--495.


\end{thebibliography}
\end{document}